\documentclass[12pt]{amsart}

\usepackage{fullpage}
\usepackage{tikz}
\usepackage{hyperref}
\usepackage[all,cmtip]{xy}
\usepackage{amsopn}
\usepackage{amsfonts}
\usepackage{amsmath}
\usepackage{amsthm}
\usepackage{amscd}
\usepackage{amssymb}
\usepackage{graphicx}
\usepackage{times} 

\makeatletter\@addtoreset {equation}{section}\makeatother

\theoremstyle{plain}
\newtheorem{definition}{Definition}[section]

\newtheorem{theorem}[definition]{Theorem}

\newtheorem{lemma}[definition]{Lemma}

\newtheorem{remark}[definition]{Remark}

\DeclareMathOperator{\sech}{sech}
\DeclareMathOperator\arctanh{arctanh}
\newcommand*\diff{\mathop{}\!\mathrm{d}}

\newenvironment{proof1}%
{\begin{trivlist} \item[]{\em Proof }}%
{\hspace*{\fill}$\rule{.3\baselineskip}{.35\baselineskip}$\end{trivlist}}

\begin{document}

\title[Orbital instability of standing waves \\
for NLS equation on Star Graphs]
{Orbital instability of standing waves \\
for NLS equation on Star Graphs}

\author{Adilbek Kairzhan}
\address{Department of Mathematics, McMaster University, Hamilton, Ontario  L8S 4K1, Canada}
\email{kairzhaa@math.mcmaster.ca}

\date{\today}

\begin{abstract}
We consider a nonlinear Schr\"{o}dinger (NLS) equation with any positive power nonlinearity on a star graph $\Gamma$ ($N$ half-lines glued at the common vertex) with a $\delta$ interaction at the vertex. The strength of the interaction is defined by a fixed value 
$\alpha \in \mathbb{R}$. In the recent works of Adami {\it et al.}, 
it was shown that for $\alpha \neq 0$ the NLS equation on $\Gamma$ admits the unique symmetric 
(with respect to permutation of edges) standing wave and that all other possible standing waves are nonsymmetric. Also, it was proved
for $\alpha<0$ that, in the NLS equation with a subcritical power-type nonlinearity, the unique symmetric standing wave is orbitally stable. 

In this paper, we analyze stability of standing waves for both $\alpha<0$ and $\alpha>0$. By extending the Sturm theory to Schr\"{o}dinger operators on the star graph, we give the explicit count of the Morse and degeneracy indices for each standing wave. 
For $\alpha<0$, we prove that all nonsymmetric standing waves in the NLS equation with any positive power nonlinearity are orbitally unstable. 
For $\alpha>0$, we prove the orbital instability of all standing waves. 
\end{abstract}

\maketitle
\section{Introduction}

The study of the nonlinear Schr\"{o}dinger (NLS) equation on different graph models is a continuously developing subject 
(see e.g. \cite{Kuchment})
motivated by various physical experiments involving wave propagation in narrow waveguides \cite{Beck,Joly1,Kuch}. In this context, graph models consisting of edges and vertices might arise as an approximation of a multi-dimensional narrow waveguides when their thickness parameters converge to zero \cite{Costa, Sobirov2}. 

In the last decade the NLS equation has been extensively studied in the context of existence and stability of standing waves on both compact (all edges are of finite lengths) and noncompact graphs. The well-known example of a compact graph is the dumbbell graph, i.e. the graph constructed by attaching two rings to a central line segment. In \cite{MP}, authors considered standing waves in the focusing NLS equation on the dumbbell graph and investigated the existence of the ground state (the standing wave that minimize the energy under the a fixed mass constraint) using methods of bifurcation theory. It has been proven that for small values of mass the ground state is given by a constant solution, whereas for larger mass values the constant solution bifurcates generating two standing waves, one of which is asymmetric and has the lowest energy at the fixed mass near the symmetry breaking bifurcation, and another standing wave is symmetric and not a ground state. The analytical results also were supported by numerical computations. Later, in \cite{Goodman}, the symmetry-preserving bifurcation described in \cite{MP} was studied in details.

The question of existence of ground states also has been raised for the NLS equation on noncompact graphs. The sufficient topological conditions for the nonexistence of ground states was given in \cite{AdamiCV} using the variational approach. In particular, it was proven that the infimum of the energy at the fixed mass $\mu$ for the NLS equation on a noncompact graph is equal to the infimum of the energy for the NLS on $\mathbb{R}$ at the same fixed mass $\mu$. Such infimum is never achieved if the noncompact graph after removal of any edge contains an edge of infinite length in every connected component with some exceptions.
In \cite{AdamiJFA}, authors gave another set of sufficient conditions, mostly based on metric properties of noncompact graphs, which guarantee the existence or nonexistence of groud states. 

Further study of the existence and stability of standing waves on graphs include, but not limited to, the works related to the cubic NLS equation on the periodic graph \cite{GilgPS}, the cubic NLS equation on the tadpole graph \cite{NojaPR} and its extension for NLS with any positive power nonlinearity \cite{NojaN}. 

Among the limitless amount of possible graph models, our particular interest is in the class of {\it star graphs}, which are constructed by gluing together $N$ half-lines (edges) at a common vertex.
Each edge can be regarded as 
$\mathbb{R}^+$, and the vertex is placed at the origin. Let $\Gamma$ represent a star graph. 
The Hilbert space on the graph $\Gamma$ is 
\begin{equation}
\label{L2}
L^2(\Gamma) = \oplus_{j=1}^N L^2(\mathbb{R}^+).
\end{equation}
Elements in $L^2(\Gamma)$ are function vectors 
$\Psi = (\psi_1, \psi_2,..., \psi_N)^T$ with each component $\psi_j \in L^2(\mathbb{R}^+)$ defined on the $j$-th edge of $\Gamma$.
We also introduce the following $L^2$-based Sobolev space on $\Gamma$
\begin{equation}
\label{Hk}
H^2(\Gamma) =  \oplus_{j=1}^N H^2(\mathbb{R}^+)
\end{equation}
equipped with appropriate boundary conditions at the vertex. 

The present work consider the NLS equation on $\Gamma$ with $\delta$-type interaction at the vertex
\begin{equation} 
\label{eq1}
i \frac{\partial \Psi}{\partial t} = - \Delta \Psi - (p+1) |\Psi|^{2p} \Psi, \quad 
\Psi \in \mathcal{D}(\Delta), \quad
t \in \mathbb{R}, \quad x \in \Gamma,
\end{equation}
where $p>0$, $\Psi=\Psi(t,x) \in \mathbb{C}^N$, $\Delta$ is the Laplacian operator with domain $\mathcal{D}(\Delta)$ acting as
$\Delta \Psi = (\psi_1'', \psi_2'',..., \psi_N'')^T$ with primes standing for derivatives in $x$,
and the nonlinearity is defined as
$|\Psi|^{2p}\Psi = 
(|\psi_1|^{2p}\psi_1, |\psi_2|^{2p}\psi_2,..., |\psi_N|^{2p}\psi_N)^T$.  The domain of the Laplacian is 
\begin{equation}
\label{H2}
\mathcal{D}(\Delta) := \left\{ \Psi \in H^2(\Gamma): \ \psi_1(0) = \dots = \psi_N(0), \
\sum_{j=1}^N  \psi_j'(0) = \alpha \psi_1(0) \right\}.
\end{equation}
where the prime denotes the right-side derivative in $x$. 
The parameter $\alpha$ incorporated in the definition of 
$\mathcal{D}(\Delta)$ defines the strenght of the vertex interaction. In the physical context, $\alpha<0$ refers to the presence of a potential well at the vertex and represents an attractive delta interaction, whereas $\alpha>0$ means the existence of a potential barrier and is associated with repulsive delta interaction. In case of $\alpha=0$, the boundary conditions in (\ref{H2}) are known as {\it Kirchhoff} and correspond to the free flow at the vertex. 

For every $p>0$, the Cauchy problem for the NLS equation (\ref{eq1}) is locally well-posed and its solutions conserve energy and mass,
see Propositions 2.1 and 2.2 in \cite{ACDD}.
The mass and energy conservations are motivated by the invariance of the NLS equation (\ref{eq1}) under the gauge transformation 
$\Psi \mapsto e^{i \theta}\Psi$ with $\theta \in \mathbb{R}$ and under the time translation $\Psi(t, \cdot) \mapsto \Psi(t+\tau, \cdot)$ with $\tau \in \mathbb{R}$, respectively. The conservation of a momentum functional is generally false since the boundary conditions at the vertex given in (\ref{H2}) break the translation symmetry in $\Gamma$. However, the momentum might be conserved under appropriate conditions. As an example, see Section 6 in \cite{KP2}, where authors considered the NLS equation (\ref{eq1}) on the star graph with even number of edges and free flow at the vertex ($\alpha = 0$).

In series of papers \cite{AdamiAH, ACDD, ACDD2}, Adami, Cacciapuoti, Finco, and Noja analyzed variational properties of standing waves in the NLS equation (\ref{eq1}) on a star graph $\Gamma$ with $N$ edges. For every $\alpha \in \mathbb{R}$, the existence and explicit formulations of standing waves were shown. In particular, for $\alpha<0$, authors found that the NLS equation (\ref{eq1}) admits the unique symmetric (with respect to permutation of edges) standing wave $\Psi_{\omega, 0}$ and that all other possible standing waves are nonsymmetric. By using the well-known stability results in \cite{Lions} and \cite{GSS}, 
in case of a subcritical nonlinearity in (\ref{eq1}) and the presence of attractive delta interaction at the vertex 
($\alpha<0$) authors proved that $\Psi_{\omega, 0}$ is orbitally stable. The proof was based on solving the (global and local) minimization problem for the NLS energy constrained by the fixed mass. For sufficiently small mass below the critical value, it was shown that a global minimizer coincides with $\Psi_{\omega, 0}$, whereas for a large mass above the critical value, $\Psi_{\omega, 0}$ is a local minimizer. 
Later, similar stability results were obtained in \cite{PG} using different approach, mostly based on the extension theory of symmetric operators. 

In case of $\alpha<0$, the energy of each nonsymmetric standing wave is higher than the energy of $\Psi_{\omega, 0}$, and such nonsymmetric standing waves are also known as {\it excited states} \cite{ACDD}. In 2014 Diego Noja published a valuable manuscript \cite{Noja} which emphasised some of the recent (2010s) results in the study of NLS on graphs and discussed various open problems. In particular, the author conjectured the orbital instability of excited states and raised the question of asymptotic stability of standing waves. The latter problem on asymptotic stability was partially answered in \cite{Li}. 

In the present work, we consider the NLS equation with any positive power nonlinearity on a star graph $\Gamma$ with the interaction at the vertex of an arbitrary strength $\alpha \neq 0$. In case of $\alpha<0$, we prove that every excited state is orbitally unstable which solves the problem conjectured in \cite{Noja}. Also, in \cite{Noja}, it has been noted that the difficulty which arises in the analysis of orbital stability/instability of excited states is to get the exact count of the Morse index for nonsymmetric standing waves. We overcome such difficulty by extending the Sturm theory to Schr\"{o}dinger operators on star graphs. Moreover, we also consider $\alpha>0$ and show the orbital instability of {\it all} standing waves in this case. 
Similar results were obtained independently in \cite{PG2} using the theory of symmetric operators, see Remark \ref{PGremark} below.

The paper is organized as follows. Section 2 gives the structure and the explicit representation of stationary states for the NLS equation (\ref{eq1}) with nonzero $\alpha$, and provides the reader with illustrations of stationary states on the star graph with $N=3$ edges. The main results on orbital instability of standing waves are introduced in Section 3 and proved in Section 4.

\section{Stationary states of the NLS equation}

Stationary states of the NLS equation (\ref{eq1}) are defined by the standing wave solutions of the form
$$
\Psi(t,x) = e^{i \omega t} \Phi_\omega(x),
$$
where each pair $(\omega, \Phi_\omega) \in \mathbb{R} \times \mathcal{D}(\Delta)$ is a real-valued solution of the stationary NLS equation
\begin{equation}
\label{stationary_NLS}
-\Delta \Phi_\omega - (p+1)|\Phi_\omega|^{2p}\Phi_\omega = -\omega\Phi_\omega.
\end{equation}
By Sobolev's embedding theorem, $\Phi_\omega \in \mathcal{D}(\Delta)$ requires $\Phi_\omega(x), \Phi_\omega'(x) \to 0$ as $x\to \infty$. Therefore, we consider only the cases with $\omega>0$ which entail the existence of the exponential decaying solutions as $x \to \infty$.

In case of $\alpha = 0$, the stationary NLS equation (\ref{stationary_NLS}) has either the unique solution or a one-parameter family of solutions, depending on the number of edges $N$ in $\Gamma$. For odd values of $N$, the solution is uniquely given by so-called ``half-solitons'', whereas for even $N$ the stationary states are represented by families of solitary waves, so-called 
``shifted states'', parametrized by a translational parameter, see \cite{ACDD}, \cite{KP1}, \cite{KP2}. The stability results on both half-solitons and shifted states are given in \cite{KP1} and \cite{KP2}.

In case of $\alpha \neq 0$, the stationary state $\Phi_{\omega, 0}$ exists if the frequency $\omega$ exceeds the lower bound 
$\omega > \frac{\alpha^2}{N^2}$ \cite{ACDD}. Other stationary states appear when $\omega$ exceeds certain bifurcation values. 
For any $K = 0,...,\left[ \frac{N-1}{2} \right]$, there exists a solution 
$\Phi_{\omega,K}$ to the stationary NLS equation (\ref{stationary_NLS}) if the condition $\omega > \frac{\alpha^2}{(N-2K)^2}$ is satisfied. The set 
$\{\Phi_{\omega, 0}, \Phi_{\omega, 1}, \Phi_{\omega, \left[ {(N-1)}/{2} \right]}\}$ represents all possible stationary states of the NLS equation (\ref{eq1}). For sufficiently large values of $\omega$, all 
$\left[ \frac{N-1}{2} \right]+1$ stationary states are present. 
The explicit representation of these stationary states is given in the next lemma proved by Theorem 4 and Remarks 5.1-5.2 in \cite{ACDD}.

\begin{lemma} 
\label{stationary}
Let $p>0$, $\alpha \in \mathbb{R} \backslash \{0\}$ and 
$K = 0,...,\left[ \frac{N-1}{2} \right]$. Then, if the condition 
$\omega > \frac{\alpha^2}{(N-2K)^2}$ is satisfied, there exists a solution
$\Phi_{\omega, K}$ to the stationary NLS equation (\ref{stationary_NLS}) 
given, up to permutations of edges, by
\begin{equation}
\label{stat_sol}
(\Phi_{\omega, K})_j(x) = \begin{cases} 
\phi_\omega(x+a_K), \quad j = 1,...,K \\
\phi_\omega(x-a_K), \quad j = K+1,...,N
\end{cases}
\end{equation}
where 
$\phi_\omega(x) = \omega^{1/2p} \sech^{1/p}(p \sqrt{\omega}x)$ and 
$a_K = \frac{1}{p\sqrt{\omega}} \arctanh \left( \frac{\alpha}{(N-2K)\sqrt{\omega}} \right)$. 
\end{lemma}

By the represention (\ref{stat_sol}), the profile of $\Phi_{\omega, K}$ on each edge of the graph $\Gamma$ is either {\it a bump} (nonmonotonic profile) or {\it a tail} (monotonic profile). The presence of such bumps or tails on the edges depends on the shift $a_K$. When $\alpha<0$, the shift value $a_K$ is negative. Therefore, 
the solution $\phi_\omega(x+a_K)$ on each edge $1,...,K$ in 
(\ref{stat_sol}) is nonmonotonic and represents a bump, whereas the solution $\phi_\omega(x-a_K)$ on each edge $K+1,...,N$ is monotonic and represents a tail. Notice that, in this case, the number of bumps in the profile of $\Phi_{\omega, K}$ is equal to $K$, and since $K<\frac{N}{2}$, there are strictly more tails than bumps. The unique symmetric standing wave $\Psi_{\omega, 0}$ described in Introduction corresponds to the stationary state $\Phi_{\omega, 0}$, while the excited states (nonsymmetric standing waves) are defined by stationary states with bumps.
As an example, if $N=3$ then $K \in \{0,1\}$ in Lemma \ref{stationary}, and there are only two possible stationary states, namely, 
$\Phi_{\omega, 0}$ and $\Phi_{\omega, 1}$, see Figure 1.

\begin{figure}[htb]
\begin{center}
\includegraphics[height=2cm,width=6cm]{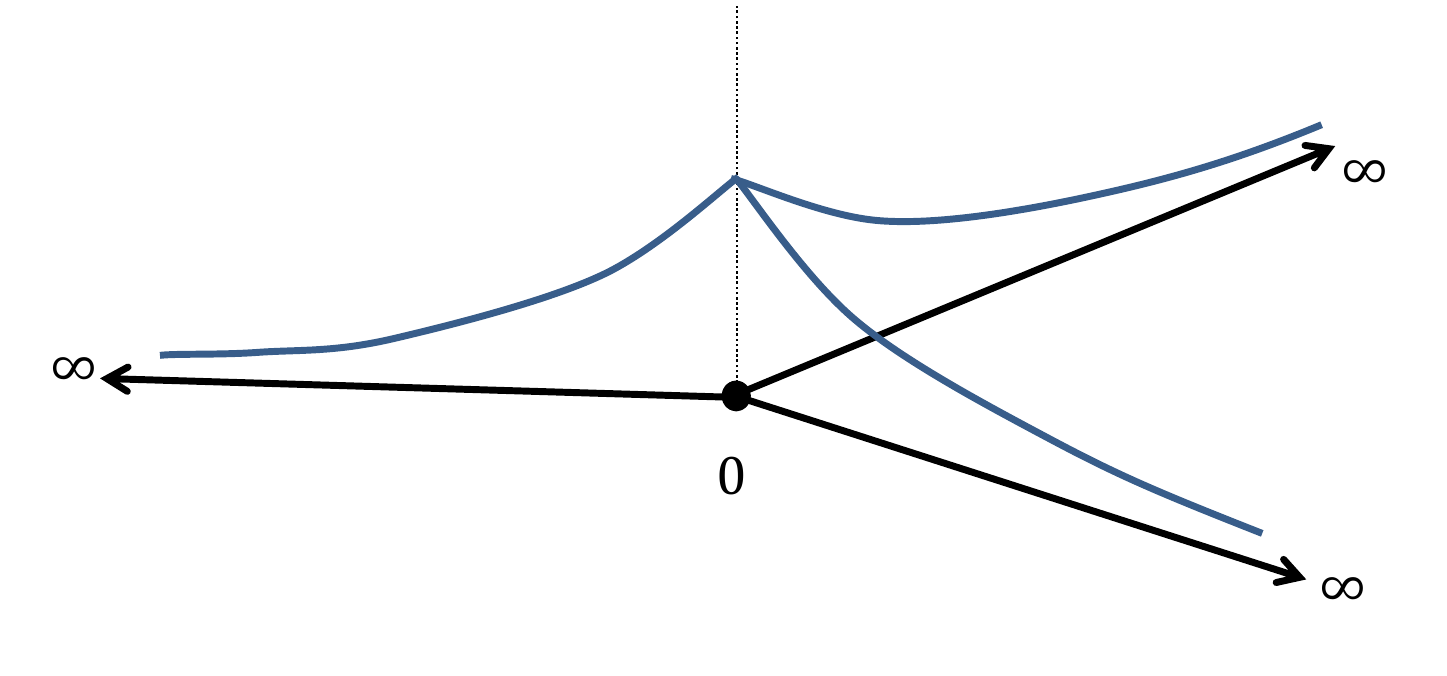}
\includegraphics[height=2cm,width=6cm]{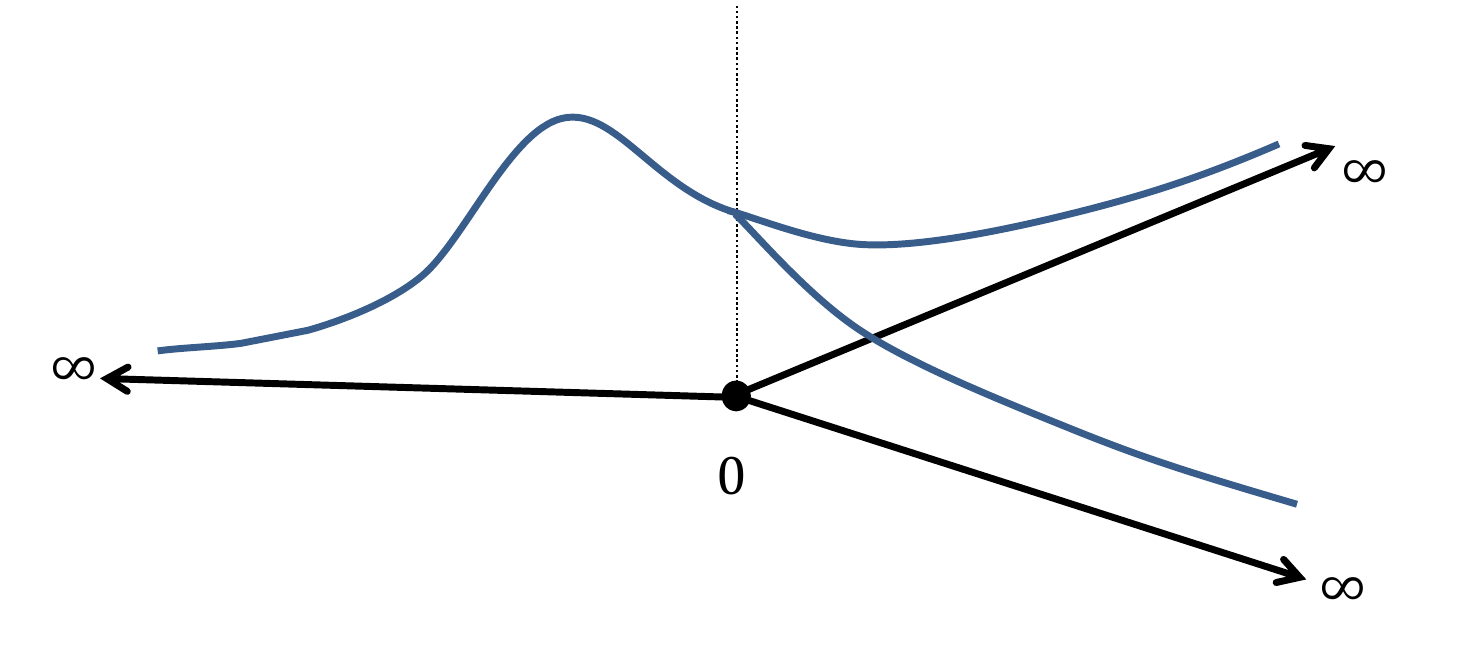}
\end{center}
\caption{Case $\alpha<0$ and $N=3$: 
$\Phi_{\omega, 0}$ has three tails and no bumps (left) and $\Phi_{\omega, 1}$ has two tails and one bump (right).}
\label{fig-1}
\end{figure}

When $\alpha>0$, the shift value $a_K$ is positive, and $K$ represents the number of tails in the profile of $\Phi_{\omega, K}$, where tails appear on the edges $1,...,K$ only according to the representation (\ref{stat_sol}). Here, in contrast to the case with negative $\alpha$, the number of tails is strictly less than the number of bumps. For $N=3$, Figure 2 illustrates the only possible stationary states $\Phi_{\omega, 0}$ and $\Phi_{\omega, 1}$.

\begin{figure}[htb]
\begin{center}
\includegraphics[height=2cm,width=6cm]{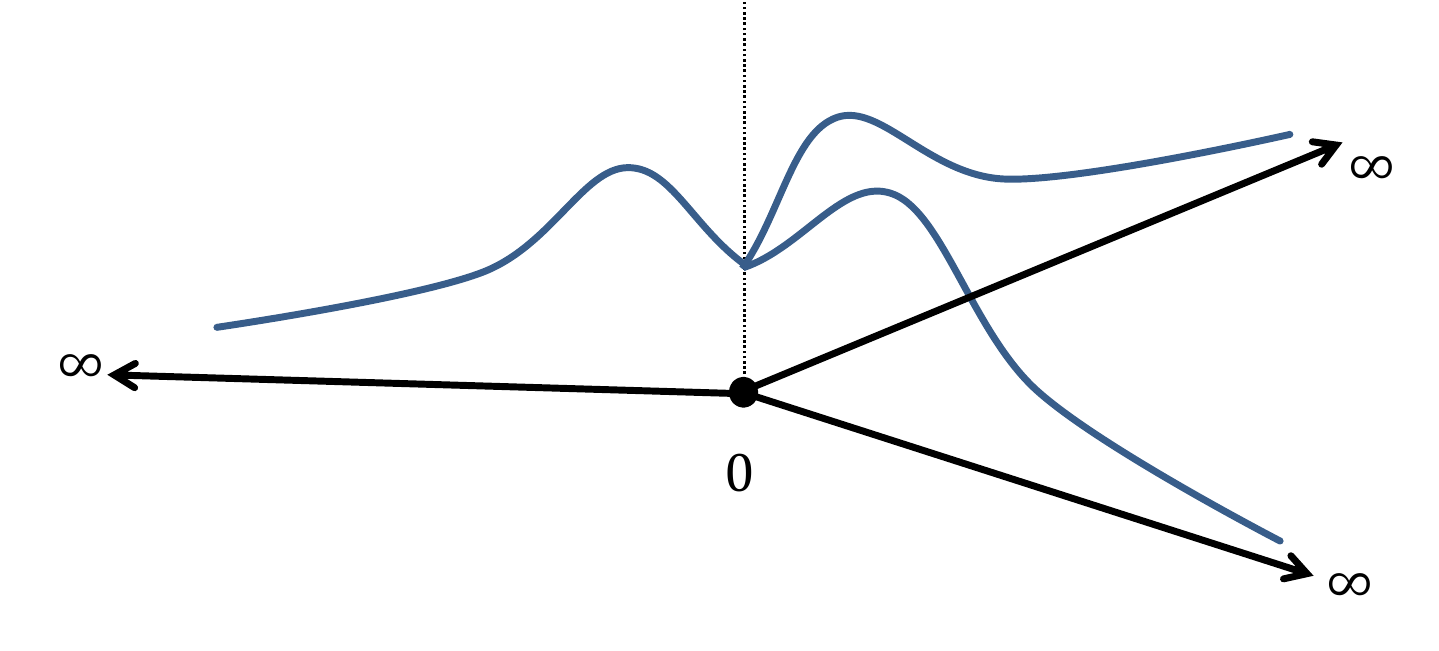}
\includegraphics[height=2cm,width=6cm]{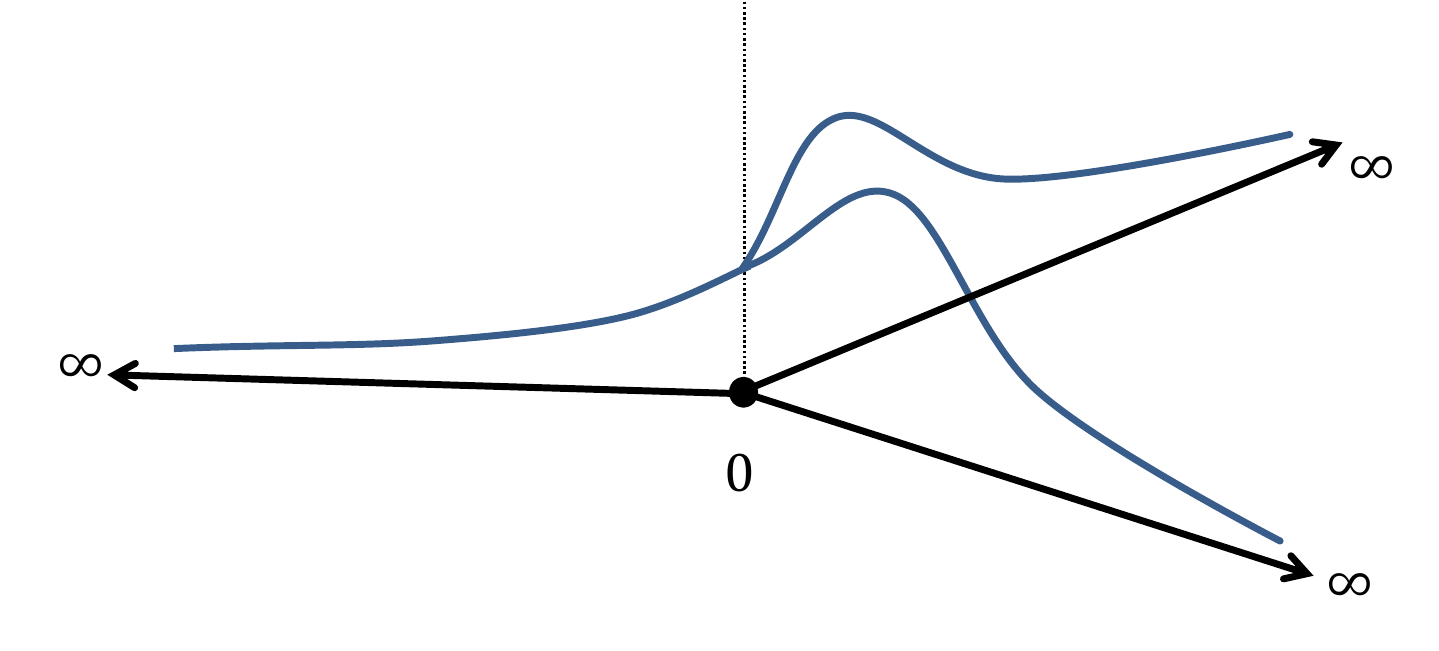}
\end{center}
\caption{Case $\alpha>0$ and $N=3$: 
$\Phi_{\omega, 0}$ has no tails and three bumps (left) and $\Phi_{\omega, 1}$ has one tails and two bumps (right)}
\label{fig-1}
\end{figure}

\section{Main results}
To investigate the stability of each stationary state $\Phi_{\omega, K}$ given by Lemma \ref{stationary}, consider its complex perturbation
$
\Psi_\omega = \Phi_{\omega} + U +iW,
$
where $U, W \in \mathcal{D}(\Delta)$ are real valued. The linearization of Hamiltonian system corresponding to the NLS equation (\ref{eq1}) implies that the time evolution of the perturbations $U$ and $W$ is
\begin{equation}
\frac{d}{dt} \left[ \begin{array}{c} U \\ W \end{array} \right] =
\left[ \begin{array}{cc} 0 & L_- \\ -L_+ & 0 \end{array} \right] \left[ \begin{array}{c} U \\ W \end{array} \right],
\end{equation}
where $L_\pm$ are Hessian operators in $L^2(\Gamma)$ with domain 
$\mathcal{D}(\Delta)$ with differentiable expression given by
\begin{eqnarray}
\label{Lplus}
L_+ & := & -\Delta + \omega - (2p+1) (p+1) \Phi_{\omega,K}^{2p} \\
L_- & := & -\Delta + \omega - (p+1) \Phi_{\omega,K}^{2p} ,
\label{Lminus}
\end{eqnarray}

Note that the operators $L_+$ and $L_-$ are self-adjoint in 
$L^2(\Gamma)$, see Theorem 1.4.4 in \cite{Kuchment}. Therefore, the spectrum $\sigma(L_\pm) \subset \mathbb{R}$ consists of the continuous and the discrete spectra only. By Weyl's theorem, see e.g. \cite{HS}, 
$\sigma_c(L_\pm) = \sigma_c(L_0)$, where $\sigma_c$ stands for the continuous spectrum and $L_0:= -\Delta + \omega$. Indeed, 
$L_\pm -L_0$ is a relatively compact perturbation to $L_0$, and Weyl's theorem is applicable. Therefore, $\sigma_c(L_\pm) = [\omega, \infty)$, and all eigenvalues of the discrete specrum $\sigma_p(L_\pm)$ are located in the interval $(-\infty, \omega)$. Define the Morse index $n(L_\pm)$ and the degeneracy index $z(L_\pm)$ as the number of negative and zero eigenvalues of $\sigma_p(L_\pm)$, respectively, with the account of their multiplicity.

The main result of this paper is that, for every $\alpha \neq 0$, we count the Morse and degeneracy indices of $L_\pm$ associated with the stationary state 
$\Phi_{\omega,K}$ for all possible $K \geq 0$, as in the following:

\begin{theorem}
\label{thm1}
Let $p>0$, $\alpha \in \mathbb{R} \backslash \{0\}$, 
$K = 0,...,\left[ \frac{N-1}{2} \right]$ and 
$\omega > \frac{\alpha^2}{(N-2K)^2}$. Let $L_+$ and $L_-$ be the Hessian operators associated with $\Phi_{\omega, K}$ and defined as in (\ref{Lplus})-(\ref{Lminus}). Then, \\
(i) $z(L_-) = 1$, $n(L_-)=0$; \\
(ii) $z(L_+)=0$ and 
$n(L_+)= \begin{cases} K+1 \quad {\rm if} \quad \alpha<0; \\ 
N-K \quad {\rm if} \quad \alpha>0.
\end{cases}$ \\
In addition, \\
$\bullet$ if $\alpha< 0$ with $K \geq 1$, $n(L_+)$ consists of two simple eigenvalues $\lambda_1<\lambda_2$ and another eigenvalue $\lambda_* \in (\lambda_1, \lambda_2)$ of multiplicity $K-1$; \\
$\bullet$ if $\alpha>0$ with $K\geq 0$, $n(L_+)$ consists a simple eigenvalue $\lambda_1$ and another eigenvalue $\lambda_{**} \in (\lambda_1, 0)$ of multiplicity $N-K-1$. 
\end{theorem}

By using the well-known instability result for the NLS equation, see Theorem 1.2 in \cite{GR}, 
the statement of Theorem \ref{thm1} is equivalent to following:
\begin{theorem}
\label{mainthm}
If $\alpha< 0$, then all standing waves $e^{i\omega t}\Phi_{\omega, K}$ with $K \geq 1$ bump(s) are spectrally (and orbitally) unstable. If $\alpha>0$, then all standing waves $e^{i\omega t}\Phi_{\omega, K}$ with $K\geq 0$ tail(s) are spectrally (and orbitally) unstable. 
\end{theorem}

For $\alpha<0$ and $K=0$, our count of Morse and degeneracy indices coincides with the count in Proposition 6.1 in \cite{ACDD} and implies the stability result given by Theorem 1 in \cite{ACDD2}:
\begin{theorem}
Let $p \in (0,2]$ and $\alpha<0$. Then, given $\omega>\frac{\alpha^2}{N^2}$, the standing wave $e^{i\omega t}\Phi_{\omega, 0}$ is orbitally stable.
\end{theorem}

\begin{remark}
\label{PGremark}
Partial results of Theorem \ref{mainthm} have been obtained in the recent work \cite{PG2} by using the extension theory of symmetric operators. 
In particular, for the case $\alpha<0$ with $p>2$, authors showed the existence of some $\omega_K^*$ such that the instability result holds for all 
$\omega \in \Big( \frac{\alpha^2}{(N-2K)^2}, \omega_K^* \Big)$. It has been noted by authors that no results were obtained for 
$\omega > \omega_K^*$. Theorem \ref{mainthm} extends these results to all $\omega \in \Big( \frac{\alpha^2}{(N-2K)^2}, \infty \Big)$. 
\end{remark}

\section{Proof of Theorem \ref{thm1}}

To prove Theorem \ref{thm1}, we compute Morse and degeneracy indices on graph $\Gamma$ by using a method developed in \cite{KP2}. Consider the following Schr\"{o}dinger equation on the real line, $x \in \mathbb{R}$,
\begin{equation} 
\label{eqofcomp}
- v''(x) + \omega v(x) - (2p+1)(p+1) \omega 
\sech^2(p\sqrt{\omega}x) v(x) = \lambda v(x),  
\quad \lambda < \omega.
\end{equation}
We are interested in the exponentially decaying solutions to (\ref{eqofcomp}) as $x \to +\infty$. The following three lemmas represent Sturm theory for the Schr\"{o}dinger equation on $\mathbb{R}$, see also \cite{KP2}.

\begin{lemma}
\label{solhyp}
For every $\lambda<\omega$, there exists a unique solution $v \in C^1(\mathbb{R})$ to equation (\ref{eqofcomp})
such that
\begin{equation}
\label{u-limit}
\lim_{x \to +\infty} v(x) e^{\sqrt{\omega-\lambda} x} = 1.
\end{equation}
The other linearly independent solution to equation (\ref{eqofcomp}) diverges as $x \to +\infty$.
Moreover, for any fixed $x_0 \in \mathbb{R}$, $v(x_0)$ is a $C^1$ function of $\lambda$ for $\lambda < \omega$ such that $\frac{v'(x_0)}{v(x_0)} \to -\infty$ as $\lambda \to -\infty$.
\end{lemma}

\begin{proof}
The proof is based on the reformulation of of the boundary-value problem (\ref{eqofcomp})--(\ref{u-limit})
as Volterra's integral equation
\begin{equation}
\label{Volterra}
v(x) = e^{-\sqrt{\omega-\lambda}x} - \frac{(2p+1)(p+1)}{\sqrt{\omega-\lambda}}
\int_x^{\infty} \sinh(\sqrt{\omega-\lambda}(x-y)) \sech^2(py) v(y) \diff y.
\end{equation}
Set $w(x;\lambda) = v(x) e^{\sqrt{\omega-\lambda}x}$ to get the following Volterra's integral equation with a bounded kernel:
\begin{equation}
\label{VolterraMod}
w(x;\lambda) = 1 + \frac{(2p+1)(p+1)}{2\sqrt{\omega-\lambda}} \int_x^{\infty}  ( 1-e^{-2\sqrt{\omega-\lambda}(y-x)}) \sech^2(py) w(y;\lambda) \diff y.
\end{equation}
By standard Neumann series and the ODE theory, the existence and uniqueness of a solution $w(\cdot;\lambda) \in C^1(\mathbb{R})$
of the integral equation (\ref{VolterraMod}) with $\lim_{x \to \infty} w(x;\lambda) = 1$ is obtained for every $\lambda < \omega$. Such construction provides with the solution $v\in C^1(\mathbb{R})$ to the differential equation (\ref{eqofcomp}) exponentially decaying as 
$x \to +\infty$. For any fixed $x_0$, $v(x_0)$ is (at least) $C^1$ function of $\lambda < \omega$ since the Volterra's integral equation (\ref{Volterra}) depends analytically on $\lambda$ for $\lambda < \omega$. 
The other linearly independent solution to the differential equation (\ref{eqofcomp}) diverges as $x\to +\infty$ due to the $x$-independent and nonzero Wronskian determinant between two solutions. 

It remains to prove that $\frac{v'(x_0)}{v(x_0)} \to -\infty$ as $\lambda \to -\infty$ for any fixed $x_0 \in\mathbb{R}$. Using the setting $w(x;\lambda) = v(x) e^{\sqrt{\omega-\lambda}x}$, we get
\begin{equation}
\label{Weyl_function}
\frac{v'(x_0)}{v(x_0)} = -\sqrt{\omega-\lambda} + 
\frac{w'(x_0;\lambda)}{w(x_0;\lambda)}. 
\end{equation}
Since $w(\cdot, \lambda) \in C^1(\mathbb{R})$ and $\lim_{x \to \infty} w(x;\lambda) = 1$, we get $w(\cdot, \lambda) \in L^\infty[x_0, \infty)$. The construction (\ref{VolterraMod}) yields that 
$\|w\|_{L^\infty[x_0, \infty)} \leq 2$ for large enough negative $\lambda$, and so, as $\lambda \to -\infty$, (\ref{VolterraMod}) implies
$$
|w(x_0;\lambda)-1| \leq \frac{C_p}{\sqrt{\omega-\lambda}} \|w\|_{L^\infty[x_0, \infty)} \leq \frac{2C_p}{\sqrt{\omega-\lambda}} \to 0,
$$
where $C_p$ is constant which depends on $p$ only. Therefore, 
\begin{equation}
\label{bound1}
w(x_0; \lambda) \to 1 \quad {\rm as} \quad \lambda \to -\infty.
\end{equation}
Differentiating the equation (\ref{VolterraMod}) in $x$, we get 
$$
w'(x;\lambda) = -(2p+1)(p+1)\int_x^{\infty}  e^{-2\sqrt{\omega-\lambda}(y-x)} \sech^2(py) w(y;\lambda) \diff y.
$$
Since the integrand in the latter expression is bounded for 
$\lambda<\omega$, for $\lambda \to -\infty$ we get
\begin{equation}
\label{bound2}
|w'(x_0; \lambda)| \leq \hat{C}_p \|w\|_{L^\infty[x_0, \infty)} \leq 2\hat{C}_p,
\end{equation}

where $\hat{C}_p$ is constant which depends on $p$ only.

Finally, by using the bounds in (\ref{bound1}) and (\ref{bound2}), the expression (\ref{Weyl_function}) implies that $\frac{v'(x_0)}{v(x_0)} \to -\infty$ as $\lambda \to -\infty$.
\end{proof}
 
\begin{lemma}
\label{Sturm}
Let $v = v(x;\lambda)$ be the solution defined by Lemma \ref{solhyp}. Assume that $v(x;\lambda_1)$
has a simple zero at $x = x_1 \in \mathbb{R}$ for some $\lambda_1 \in (-\infty,\omega)$.
Then, there exists a unique $C^1$ function
$\lambda \mapsto x_0(\lambda)$ for $\lambda$ near $\lambda_1$ such that $v(x;\lambda)$ has a simple
zero at $x = x_0(\lambda)$ with $x_0(\lambda_1) = x_1$ and $x_0'(\lambda_1) > 0$.
\end{lemma}

\begin{proof}
By the previous lemma, $v$ is a $C^1$ function of $x$ and $\lambda$ for every $x \in \mathbb{R}$
and $\lambda \in (-\infty,\omega)$. Since $x_1$ is a simple zero of $v(x;\lambda_1)$, we get
$\partial_x v(x_1;\lambda_1) \neq 0$, and this allows us to use the implicit function theorem
to get a unique $C^1$ function $\lambda \mapsto x_0(\lambda)$ for $\lambda$ near $\lambda_1$
such that $v(x;\lambda)$ has a simple
zero at $x = x_0(\lambda)$ with $x_0(\lambda_1) = x_1$.
It remains to show that $x_0'(\lambda_1) > 0$.

Notice that $v(x_0(\lambda);\lambda) = 0$ implies 
\begin{equation}
\label{root-der-eq}
\partial_\lambda v(x_0(\lambda); \lambda)\Big|_{\lambda=\lambda_1} = 
\partial_x v(x_1;\lambda_1) x_0'(\lambda_1) + \partial_{\lambda} v(x_1;\lambda_1) = 0.
\end{equation}
Let us denote $\tilde{v}(x) = \partial_{\lambda} v(x;\lambda_1)$. Then, differentiating
equation (\ref{eqofcomp}) in $\lambda$, we have the
differential equation for $\tilde{v}$ on $\mathbb{R}$:
\begin{equation} \label{der-v-eq}
- \tilde{v}''(x) + \tilde{v}(x) - (2p+1)(p+1) \sech^2(px) \tilde{v}(x) =
\lambda_1 \tilde{v}(x) + v(x;\lambda_1), \quad \lambda < \omega.
\end{equation}
By using the Volterra's integral equation as in Lemma \ref{solhyp} above, we have that the function $\tilde{v}$ is $C^1$ in $x$ and decays to zero as $x \to \infty$. We multiply (\ref{der-v-eq}) by $v(x;\lambda_1)$,
integrate by parts on $[x_1,\infty)$, and use equation (\ref{eqofcomp}) to get
\begin{equation} \label{der-vv-eq}
- \partial_x v(x_1;\lambda_1) \tilde{v}(x_1) = \int_{x_1}^{\infty} v(x;\lambda_1)^2 dx.
\end{equation}
Finally, by using the expressions (\ref{root-der-eq}) and (\ref{der-vv-eq}) we have
\begin{equation}
\label{zeroshift}
(\partial_x v(x_1;\lambda_1) )^2 x_0'(\lambda_1) = \int_{x_1}^{\infty} v(x;\lambda_1)^2 dx > 0
\end{equation}
which, together with $\partial_x v(x_1;\lambda_1) \neq 0$, implies that $x_0'(\lambda_1) > 0$.
\end{proof}

The following result represents the 
\begin{lemma}
\label{lem-Sturm}
Let $v$ be the solution defined in Lemma \ref{solhyp}. If $v(0) = 0$ (resp. $v'(0) = 0$) for some $\lambda_0 < \omega$,
then the eigenfunction $v$ to the Schr\"{o}dinger equation (\ref{eqofcomp}) is
an odd (resp. even) function on $\mathbb{R}$, hence $\lambda_0$ is an eigenvalue of
the associated Schr\"{o}dinger operator defined in $L^2(\mathbb{R})$.
There exists exactly one $\lambda_0 < 0$ corresponding to $v'(0) = 0$ and a simple eigenvalue $\lambda_0 = 0$
corresponding to $v(0) = 0$, all other such points $\lambda_0$ are located in $(0,\omega)$ and are bounded away from zero.
\end{lemma}

\begin{proof}
The uniqueness of the solution $v$ in Lemma \ref{solhyp} and the reversibility of the Schr\"{o}dinger equation (\ref{eqofcomp}) with
respect to the transformation $x \mapsto -x$ yields the existence of even and odd eigenfunctions on $\mathbb{R}$. The count of eigenvalues follows by Sturm's Theorem since
the odd eigenfunction for the eigenvalue $\lambda_0 = 0$,
$$
\phi'(x) = - \sech^{\frac{1}{p}}(px) \tanh(px)
$$
has one zero on the infinite line. Hence, $\lambda_0 = 0$ is the second eigenvalue of
the Schr\"{o}dinger equation (\ref{eqofcomp}) with exactly one simple negative eigenvalue $\lambda_0 < 0$
that corresponds to an even eigenfunction.
\end{proof}

Further in the paper, we use $\lambda_0$ to denote the eigenvalue from Lemma \ref{lem-Sturm} with positive even eigenfunction. The Figure 3 illustrates Lemmas \ref{solhyp}-\ref{lem-Sturm} and shows profiles of the solution $v$ satisfying (\ref{eqofcomp}) and (\ref{u-limit}) for values of $\lambda$ in the interval $(-\infty, 0)$.

\begin{figure}[htb]
\begin{center}
\includegraphics[height=2cm,width=6cm]{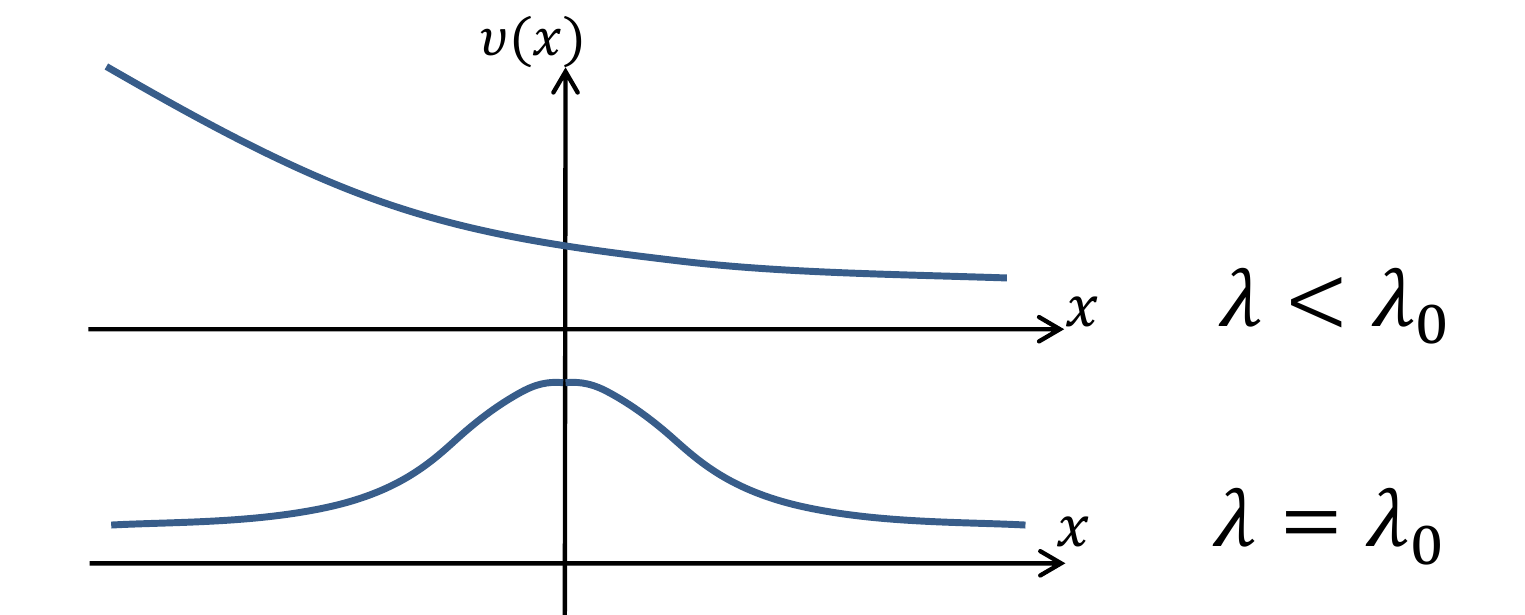}
\includegraphics[height=2cm,width=6cm]{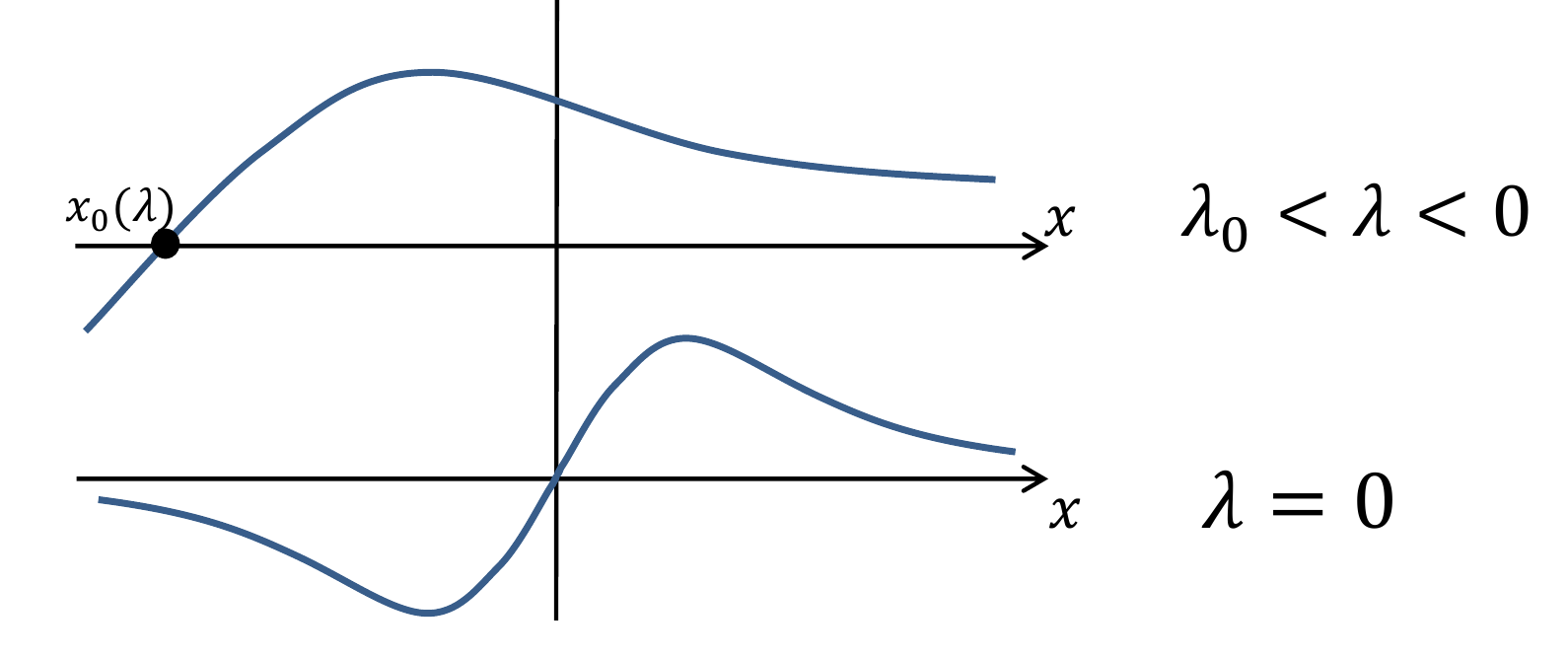}
\end{center}
\caption{The solution $v$ of the differential equation (\ref{eqofcomp}) with 
the limit (\ref{u-limit}) versus $x$ for several values of $\lambda$}
\label{fig-3}
\end{figure}

Next two lemmas provide us with the useful tools to compute the Morse and degeneracy indices of the operator $L_+$ at the state $\Phi_{\omega,K}$. 
\begin{lemma}
\label{lem-eig}
Let $\alpha \in \mathbb{R}\backslash \{0\}$, $K \geq 0$, $v$ be the solution to (\ref{eqofcomp}) given by Lemma \ref{solhyp}, and $L_+$ be the Hessian operator (\ref{Lplus}) associated with $\Phi_{\omega,K}$. 
Then,
$\lambda \in (-\infty, \omega)$ is an eigenvalue of $\sigma_p(L_+)$ if and only if at least one of the following conditions holds: \\
(a) $v(a_K) = 0$ with $K \geq 1$, \\
(b) $v(-a_K) = 0$,\\
(c) $Kv'(a_K)v(-a_K)+(N-K)v(a_K)v'(-a_K)-\alpha v(a_K)v(-a_K)=0$.\\
Moreover, $\lambda \in \sigma_p(L_+)$ has mutliplicity $K-1$ in the case (a), $N-K-1$ in the case (b), and is simple in the case (c). If $\lambda$ satisfies several cases, then its multiplicity is the sum of the multiplicities of each case.
\end{lemma}

\begin{proof}
First, assume $K \geq 1$. Denote $a: = a_K$. Let $\lambda$ be the eigenvalue of $L_+$ with the eigenvector 
$U=(u_1,...,u_N)^T \in \mathcal{D}(\Delta)$. Then, in the eigenvalue problem 
$L_+ U = \lambda U$, each component can be written as the second order differential equation
\begin{equation}
\label{eqofcompLplus}
-u_j''(x) + \omega u_j(x) - (2p+1)(p+1)\omega \sech^2 \big(p\sqrt{\omega}
(x+(-1)^{m_j} a) \big) u_j(x) = \lambda u_j(x), \quad x\in (0, \infty),
\end{equation}
where 
$$m_j= \begin{cases} 0 \quad {\rm for} \quad j=1,...,K, \\ 
1 \quad {\rm for} \quad j=K+1,...,N. \end{cases}
$$
The substitution $u_j(x) = c_j v(x + (-1)^{m_j}a)$ with coefficient $c_j$ transforms (\ref{eqofcompLplus}) into (\ref{eqofcomp}). By Sobolev embedding of $H^2(\mathbb{R}^+)$ into $C^1(\mathbb{R}^+)$, $u_j(x) \to 0$ and $u_j'(x) \to 0$ as $x \to +\infty$ for each $j=1,...,N$. To satisfy the latter condition, we let $v$ to be the solution to (\ref{eqofcomp}) given by Lemma \ref{solhyp}, and so
\begin{equation}
\label{components}
u_j(x) = \begin{cases} c_j v(x+a), & j=1,\dots,K, \\
c_j v(x-a), & j=K+1,\dots,N.
\end{cases}
\end{equation}
The boundary conditions for $U \in \mathcal{D}(\Delta)$ in (\ref{H2})
imply the homogeneous linear system on the coefficients
\begin{equation}
\label{lin-system-1}
c_1 v(a) = \dots = c_K v(a) =c_{K+1} v(-a) = \dots = c_N v(-a), \quad 
{\rm and}
\end{equation}
\begin{equation}
\label{lin-system-2}
\sum_{j=1}^{K}c_j v'(a)+ \sum_{j=K+1}^N c_j  v'(-a) = \alpha c_N v(-a).
\end{equation}
The associated matrix is
\begin{equation}
{\small M=
\left( \begin{array}{ccccccccc}
v(a) & -v(a) & 0 & \dots & 0 & 0 & \dots & 0 & 0 \\
v(a) & 0 & -v(a) & \dots & 0 & 0 & \dots & 0 & 0 \\
\vdots & \vdots & \vdots & \ddots & \vdots & \vdots & \ddots & \vdots & \vdots \\
v(a) & 0 & 0 & \dots & -v(a) & 0 & \dots & 0 & 0 \\
v(a) & 0 & 0 & \dots & 0 & -v(-a) & \dots & 0 & 0 \\
\vdots & \vdots & \vdots & \ddots & \vdots & \vdots & \ddots & \vdots & \vdots \\
v(a) & 0 & 0 & \dots & 0 & 0 & \dots & 0 & v(-a) \\
v'(a) & v'(a) & v'(a) & \dots & v'(a) & v'(-a) & \dots & v'(-a) &
v'(-a) - \alpha v(-a) \\
\end{array} \right)
}
\end{equation}
Doing elementary column operations, we can obtain a lower triangular matrix with the determinant
$$
{\rm det} M = v(a)^{K-1} v(-a)^{N-K-1} 
\left[ Kv'(a)v(-a)+(N-K)v'(-a)v(a) - \alpha v(a)v(-a) \right].
$$
Therefore, $U \neq 0$ is the eigenvector of $L_+$ for the eigenvalue $\lambda \in (-\infty,\omega)$
if and only if ${\rm det} M = 0$, or equivalently, at least one of the conditions (a), (b), (c) is true. 
The multiplicity of $\lambda$ in cases 
(a)-(c) comes directly from the linear system (\ref{lin-system-1})--(\ref{lin-system-2}).

For $K=0$, the boundary conditions (\ref{lin-system-1})-(\ref{lin-system-2}) do not contain terms $v(a)$ and $v'(a)$, and the determinant of the associated matrix $M$ becomes 
$$
{\rm det} M = v(-a)^{N-1} \left[Nv'(-a) - \alpha v(-a) \right].
$$
Then, ${\rm det} M = 0$ if and only if at least one of the conditions (b), (c) is true. 
\end{proof}

\begin{lemma}
\label{lastlemma}
Let $\alpha \in \mathbb{R} \backslash \{0\}$, $K \geq 0$ and $v$ be the solution to (\ref{eqofcomp}) given by Lemma \ref{solhyp}. Consider the  function of $\lambda$ as
\begin{equation}
\label{F}
F(\lambda) := K\frac{v'(a_K; \lambda)}{v(a_K;\lambda)} + 
(N-K) \frac{v'(-a_K; \lambda)}{v(-a_K; \lambda)}: (-\infty, 0] \to \mathbb{R}. 
\end{equation}
Then, the following hold: \\
$\bullet$ $v(|a_K|; \lambda)>0$ for all $\lambda \in (-\infty, 0]$, and there exists unique $\lambda_* \in (-\infty, 0]$ such that $v(-|a_K|;\lambda) = 0$. Moreover, $\lambda_*<0$. \\
$\bullet$ if $\alpha<0$ with $K \geq 1$, then $F(\lambda) = \alpha$ has exactly two solutions $\lambda_1 < \lambda_2$. Moreover, 
$\lambda_1 < \lambda_* < \lambda_2 < 0$. \\
$\bullet$ if $\alpha>0$ with $K \geq 0$ then $F(\lambda) = \alpha$ has the unique root $\lambda_1$, and, moreover, $\lambda_1<\lambda_*$.
\end{lemma}

\begin{proof}
Since $v$ is the nonzero solution of the second order differential equation (\ref{eqofcomp}), it has only simple zeros which, according to Lemma \ref{Sturm}, are monotonically increasing functions of $\lambda$. At $\lambda = \lambda_0$, by Lemma \ref{lem-Sturm}, we have positive even $v(x; \lambda_0)$ exponentially decaying as $|x| \to \infty$. Therefore, $v(x;\lambda)$ has the only zero $x_0(\lambda)$ which bifurcates from $x=-\infty$ at $\lambda=\lambda_0$ and moves strictly monotonically towards $x=0$ as $\lambda \to 0$ with $x_0(0)=0$, see Figure \ref{fig-3}. As a result, for $\lambda\leq \lambda_0$, $v(x;\lambda)$ is positive on the entire real line, whereas, for $\lambda \in (\lambda_0, 0]$, $v(x;\lambda)$ is positive for every 
$x \in (x_0(\lambda), \infty)$ and $v(x_0(\lambda);\lambda)=0$. We denote $\lambda$ satisfying $x_0(\lambda) = -|a_K|$ as $\lambda_*$. Since $a_K \neq 0$, then $\lambda_*<0$.
 The uniqueness of $\lambda_*$ is guaranteed by the monotonicity of $x_0(\lambda)$. This proves the first assertion of this Lemma. 

By Lemma \ref{solhyp}, $v(\pm a_K;\lambda)$ is a $C^1$ function of $\lambda$ for $\lambda \leq 0$. Therefore, using the first assertion proven above, for $\alpha<0$ with $K\geq 1$ and for $\alpha>0$ with $K\geq 0$, 
$F(\lambda)$ is $C^1((-\infty, 0])\backslash \{\lambda_*\})$ and has a simple pole at $\lambda = \lambda_*$. 

To investigate the behaviour of the function $F$, we first show that $F$ is a monotonically increasing function. 
Differentiating the equation (\ref{eqofcomp}) in $\lambda$, multiplying it by $v$ and integrating by parts on $[c, \infty]$ for some $c \in \mathbb{R}$, we get
$$
P(c): = \frac{\partial_\lambda v'(c) v(c) - v'(c) \partial_\lambda(c)}{v^2(c)} = \frac{1}{v^2(c)} \int_{c}^\infty v^2(x) dx > 0 \quad {\rm if} \quad v(c) \neq 0.
$$
Therefore, $F'(\lambda) = K P(a_K) + (N-K) P(-a_K) > 0$ for all $\lambda \in (-\infty, 0]\backslash \{\lambda_*\}$. 

By Lemma \ref{solhyp}, for every $c \in \mathbb{R}$, $\lim_{\lambda \to -\infty} \frac{v'(c;\lambda)}{v(c;\lambda)} = -\infty $. 
Then, taking $c = \pm a_K$, we have that 
$\lim_{\lambda \to -\infty} F(\lambda) = -\infty$. 
 By Lemma \ref{Sturm} on the monotonicity of a simple zero of $v$, the behaviour of $F(\lambda)$ around the point of singularity $\lambda_*$ is given by
$$
\lim_{\lambda \to \lambda_*^-} F(\lambda) = +\infty \quad {\rm and} 
\quad \lim_{\lambda \to \lambda_*^+} F(\lambda) = -\infty.
$$
At $\lambda=0$, the unique solution $v = v(x;0)$ of (\ref{eqofcomp}) in Lemma \ref{solhyp} is known to be $v(x) = -C\phi'_\omega(x)$, where $\phi_\omega$ is given by Lemma \ref{stationary} and $C = 2^{-1/p}\omega^{-(1+p)/2p}$. Then, using the explicit formulations of $v$ and $a_K$, direct computations give 
$$
F(0) = \frac{p(\alpha^2-(N-2K)^2 \omega)}{\alpha} + \alpha. 
$$
Since $\omega> \frac{\alpha^2}{(N-2K)^2}$, then 
$p(\alpha^2-(N-2K)^2 \omega)<0$. Hence, for $\alpha<0$ we have $F(0)>\alpha$, whereas for $\alpha>0$ we have $F(0)<\alpha$. As a result,  for every $\alpha \neq 0$, the equation $F(\lambda) = \alpha$ has a unique root $\lambda_1 \in (-\infty, \lambda_*)$. Moreover, for $\alpha<0$, there is an additional root $\lambda_2$ which is unique in $(\lambda_*, 0)$. This proves the remaining assertions of this Lemma.
\end{proof}

\begin{remark}
\label{lastremark}
In case of $\alpha<0$ with $K=0$, similar steps as in the proof of Lemma \ref{lastlemma} lead to the conclusion that $F(\lambda)$ is $C^1(-\infty, 0])$ and monotonically increasing with $\lim_{\lambda \to -\infty} F(\lambda) = -\infty$ and $F(0)>\alpha$. Then, $F(\lambda)=\alpha$ has the unique root $\lambda_1 \in (-\infty, 0)$. Consequently, $n(L_+) = 1$ and Theorem \ref{thm1} follows by a standard application of the orbital stability Theorem, see \cite{ACDD} and \cite{ACDD2}.
\end{remark}

\begin{proof1}{\em of Theorem \ref{thm1}.}
The proof of the Part (i) is based on the same arguments as in 
Proposition 6.1 in \cite{ACDD}, Theorem 4.1 in \cite{PG} and Lemma 3.1 in \cite{KP1}. Thus, we only concerned about the proof of the Part (ii).

Let $\hat \lambda \in \sigma_p(L_+) \cap (-\infty, 0]$ be a nonpositive eigenvalue of $\sigma_p(L_+)$ with the eigenvector 
$U \in \mathcal{D}(\Delta)$. Then, by Lemma \ref{lem-eig} at least one of the conditions {\it (a), (b), (c)} must be satisfied by $v(x; \hat \lambda)$. 

Consider $\alpha<0$ with $K\geq 1$ or $\alpha>0$ with $K\geq 0$.  Recall that $a_K<0$ for negative $\alpha$, and $a_K>0$ for positive $\alpha$. Then, for $\alpha<0$ with $K\geq 1$, by Lemma \ref{lastlemma}, the part (a) of Lemma \ref{lem-eig} is satisfied for unique $\lambda_* \in (-\infty, 0]$ and the part (b) is never true. For $\alpha>0$ with $K\geq 0$, the part (a) is never true and the part (b) is satisfied for unique $\lambda_{**} \in (-\infty, 0]$. 

It remains to consider the part (c) of Lemma \ref{lem-eig}, namely, to find all values $\hat \lambda \in (-\infty, 0]$ such that $v(x) = v(x; \hat \lambda)$ will satisfy
\begin{equation}
\label{part_c}
Kv'(a_K)v(-a_K)+(N-K)v'(-a_K)v(a_K) - \alpha v(a_K)v(-a_K) = 0
\end{equation}
Since $v(|a_K|) \neq 0$, and $v'(-|a_K|) \neq 0$ if $v(-|a_K|)=0$, the eigenvalue $\lambda = \lambda_*$ (resp. $\lambda = \lambda_{**}$) is not a solution of (\ref{part_c}). Therefore, all solutions $\hat \lambda$ of (\ref{part_c}) coincide with all solutions of $F(\lambda)=\alpha$, where $F$ is given by (\ref{F}). The last two assertions of Lemma \ref{lastlemma} complete the proof of Theorem \ref{thm1}.

For the case $\alpha<0$ with $K=0$, Remark \ref{lastremark} implies that both parts (a) and (b) of Lemma \ref{lem-eig} are never true, and the remaining part (c) has a unique root $\lambda_1 \in (-\infty, 0]$ with $\lambda_1<0$. Thus, 
$z(L_+)=0$ and $n(L_+) = 1$ are proved. 

\end{proof1}

\section*{Acknowledgement}
The author is grateful to Dmitry Pelinovsky for enlightening discussions and helpful suggestions during the preparation of this work.

\bibliographystyle{amsplain}

\end{document}